\documentclass[12pt,reqno]{amsart}
\usepackage[colorlinks,citecolor=red,pagebackref,hypertexnames=false]{hyperref}

\makeatletter
\def\@tocline#1#2#3#4#5#6#7{\relax
  \ifnum #1>\c@tocdepth 
  \else
    \par \addpenalty\@secpenalty\addvspace{#2}%
    \begingroup \hyphenpenalty\@M
    \@ifempty{#4}{%
      \@tempdima\csname r@tocindent\number#1\endcsname\relax
    }{%
      \@tempdima#4\relax
    }%
    \parindent\z@ \leftskip#3\relax \advance\leftskip\@tempdima\relax
    \rightskip\@pnumwidth plus4em \parfillskip-\@pnumwidth
    #5\leavevmode\hskip-\@tempdima
      \ifcase #1
       \or\or \hskip 1em \or \hskip 2em \else \hskip 3em \fi%
      #6\nobreak\relax
    \dotfill\hbox to\@pnumwidth{\@tocpagenum{#7}}\par
    \nobreak
    \endgroup
  \fi}
\makeatother

\usepackage{tikz,amsthm,amsmath,amstext,amssymb,amscd,epsfig,euscript, mathrsfs, dsfont,pspicture,multicol,graphpap,graphics,graphicx,times,enumerate,subfig,sidecap,wrapfig,color,pict2e,wrapfig}

\usepackage{enumerate}

 \numberwithin{equation}{section}

\def\bB{{\mathbb{B}}}

\def\bR{{\mathbb{R}}}
\def\bS{{\mathbb{S}}}

\def\bN{{\mathbb{N}}}

\def\cH{{\mathscr{H}}}

\def\ve{\varepsilon}

\renewcommand{\d}{{\partial}}

\DeclareMathOperator{\diam}{diam}
\def\co{\mathop\mathrm{co }}						
\def\dim{\mathop\mathrm{dim}} 					
\def\dist{\mathop\mathrm{dist}} 						

\newcommand{\ps}[1]{\left( #1 \right)}

\newcommand{\ck}[1]{\left\{#1 \right\}}

\def\XXint#1#2#3{{\setbox0=\hbox{$#1{#2#3}{\int}$ }
\vcenter{\hbox{$#2#3$ }}\kern-.58\wd0}}

\theoremstyle{plain}
\newtheorem{theorem}{Theorem}

\newtheorem{corollary}[theorem]{Corollary}

\newtheorem{lemma}[theorem]{Lemma}

\theoremstyle{definition}

\newtheorem{definition}[theorem]{Definition}
\newtheorem{remark}[theorem]{Remark}
\newtheorem*{thmi}{Theorem I}
\newtheorem*{thmii}{Theorem II}

\numberwithin{equation}{section}
\numberwithin{theorem}{section}

\newcommand\eqn[1]{\eqref{e:#1}}

\newcommand\Lemma[1]{Lemma \ref{l:#1}}

  \DeclareFontFamily{U}{mathb}{\hyphenchar\font45} 
\DeclareFontShape{U}{mathb}{m}{n}{
      <5> <6> <7> <8> <9> <10> gen * mathb
      <10.95> mathb10 <12> <14.4> <17.28> <20.74> <24.88> mathb12
      }{}
\DeclareSymbolFont{mathb}{U}{mathb}{m}{n}

\DeclareMathSymbol{\toitself}      {3}{mathb}{"FD}  

\begin{document}

\title{Tangents, rectifiability, and corkscrew domains}

\author{Jonas Azzam}
\address{School of Mathematics, University of Edinburgh, JCMB, Kings Buildings,
Mayfield Road, Edinburgh,
EH9 3JZ, Scotland.}
\email{j.azzam "at" ed.ac.uk}
\thanks{The author was supported by grants ERC grant 320501 of the European Research Council (FP7/2007-2013)}

\maketitle

\begin{abstract}
In a recent paper, Cs\"ornyei and Wilson prove that curves in Euclidean space of $\sigma$-finite length have tangents on a set of positive $\cH^{1}$-measure. They also show that a higher dimensional analogue of this result is not possible without some additional assumptions. In this note, we show that if $\Sigma\subseteq \bR^{d+1}$ has the property that each ball centered on $\Sigma$ contains two large balls in different components of $\Sigma^{c}$ and $\Sigma$ has $\sigma$-finite $\mathscr{H}^{d}$-measure, then it has $d$-dimensional tangent points in a set of positive $\cH^{d}$-measure. As an application, we show that if the dimension of harmonic measure for an NTA domain in $\bR^{d+1}$ is less than $d$, then the boundary domain does not have $\sigma$-finite $\cH^{d}$-measure.

We also give shorter proofs that Semmes surfaces are uniformly rectifiable and, if $\Omega\subseteq \mathbb{R}^{d+1}$ is an exterior corkscrew domain whose boundary has locally finite $\mathscr{H}^{d}$-measure, one can find a Lipschitz subdomain intersecting a large portion of the boundary.

\end{abstract}

\tableofcontents

\section{Introduction}

In \cite{CW15}, Cs\"ornyei and Wilson show that curves of $\sigma$-finite $\cH^{1}$-measure in Euclidean space have tangents on a set of positive $\cH^{1}$-measure. 
For the definition of a tangent, see Definition \ref{d:tangent} below. They also show that the same result does not hold for higher dimensional surfaces by constructing a $d$-dimensional topological sphere $\Sigma\subseteq \bR^{d+1}$ with $\cH^{d}(\Sigma)<\infty$ but no $d$-dimensional tangents anywhere. Their example still contains a piece of a Lipschitz graph, and thus, for almost every $\xi$ in this set, it has {\it approximate $d$-dimensional tangents}, meaning $\liminf_{r\downarrow 0} \cH^{d}(B(\xi,r)\cap \Sigma)/r^{d}>0$ and there is a $d$-plane $V$ so that for all $t>0$, $\lim_{r\rightarrow 0} \cH^{d}(\{z\in B(\xi,r)\cap \Sigma: \dist(z,L)>t|z-\xi|\})/r^{d}=0$ (see Chapter 15 of \cite{Mattila} for more on tangents). 

\begin{definition}
For $C\geq 2$, a closed set $\Sigma\subseteq \bR^{d+1}$ satisfies the {\it $C$-two-ball condition} if for each $\xi\in \Sigma$ and $r\in (0,\diam \Sigma)$, there are two balls of radius $r/C$ contained in $B(\xi,r)$ in two different components of $\bR^{d+1}\backslash \Sigma$.
\end{definition}

With this extra condition, we obtain a generalization of the above result.

\begin{thmi}\label{t:MW}
If $\Sigma\subseteq \bR^{d+1}$ satisfies the two-ball condition and has $\sigma$-finite $\cH^{d}$-measure, then for any ball $B$ centered on $\Sigma$, the set of tangent points in $B\cap \Sigma$ has positive $\cH^{d}$-measure.
\end{thmi}
\def\TheoremI{Theorem \hyperref[t:MW]{I} }

It would be interesting to find a higher codimensional analogue of the above result, perhaps a variant of the generalized Semmes surfaces introduced by David, see page 107 of \cite{Dav88}.\\

As an application of this result, recall the dimension of a measure $\omega$ is 
\begin{multline*}
\dim \omega=\inf\{t:\mbox{ there is $E$ such that $\cH^{t}(E)=0$ and } \\
\mbox{for all compact sets $K$, $\omega(K)=\omega(K\cap E)$}\}.\end{multline*}

Makarov showed that the harmonic measure $\omega$ for a simply connected planar domain has $\dim \omega=1$ \cite{Mak85}. However, Wolff showed it was possible in $\bR^{3}$ to have NTA domains topologically equivalent to the sphere so that the associated harmonic measure has dimension larger or less than 2 (see also \cite{LVV05} for generalizations to higher dimensions). Badger used his main result in \cite[Theorem 5.1]{Bad12} to show that, if the harmonic measure $\omega$ for an NTA domain in $\bR^{d+1}$ had $\dim \omega<d$, then necessarily $\cH^{d}|_{\d\Omega}$ is locally infinite. Using \TheoremI, we obtain the following improvement.

\begin{corollary}
Let $\Omega\subseteq \bR^{d+1}$ be a corkscrew domain and $\omega$ its harmonic measure. Suppose $\dim \omega < d$. Then $\cH^{d}|_{\d\Omega}$ is not $\sigma$-finite.
\end{corollary}

If $\cH^{d}|_{\d\Omega}$ were $\sigma$-finite, then it would have tangents on a set $K\subseteq \d\Omega$ of positive $\cH^{d}$-measure by \TheoremI. Theorem III in \cite{AAM16} says that $\cH^{d}\ll \omega\ll \cH^{d}$ on the set of interior cone points for $\Omega$, and since tangent points are also cone points, it follows from  that $\cH^{d}\ll \omega\ll \cH^{d}$ on $K$. Since $\dim \omega <d$, we can find $E$ so that $\cH^{d}(E)=0$ and $\omega(E\cap K)=\omega(K)>0$, but the mutual absolute continuity would imply $0=\cH^{d}(E)\geq  \cH^{d}(E\cap K)>0$, a contradiction, and thus proves the corollary.\\

The techniques for proving \TheoremI can also be used to give a criterion for when a domain has Lipschitz subdomains intersecting a large portion of the boundary, which produces a shorter proof of a result from uniform rectifiability.

\begin{definition} A closed set $E\subseteq \bR^{n}$ is {\it $d$-uniformly rectifable} if 
\begin{enumerate}
\item $E$ is {\it $d$-Ahlfors regular}, meaning there is $A>0$ so that
\begin{equation}\label{e:AR}
  r^{d}/A\leq \cH^{d}(B(\xi,r)\cap E) \leq Ar^{d} \mbox{ for }\xi\in E, r\in (0,\diam E),
\end{equation}
  \item {\it $E$ has big pieces of Lipschitz images}, meaning there are constants $L,c>0$ so  for all $\xi\in E$ and $r\in (0,\diam E)$, there is a Lipschitz map $f:\bR^{d}\rightarrow \bR^{n}$ $L$-Lipschitz and $\cH^{d}(f(\bR^{d})\cap B(\xi,r))\geq cr^{d}$.
  \end{enumerate}
\end{definition}

These sets were introduced by David and Semmes in \cite{DS} in the context of singular integrals, and it is an interesting problem to isolate simple geometric criteria that guarantee uniform rectifiability. Below is one such criterion due to Semmes \cite{Semmes89}:

\begin{definition}
A $d$-Alhfors regular set $E\subseteq \bR^{d+1}$ satisfying the two-ball condition is called a {\it Semmes surface}.
\end{definition}
In \cite{Dav88}, David showed that Semmes surfaces are unfiformly rectifiable as well as certain higher codimensional generalizations. Since then, other proofs have been developed and in much more generality, see for example  \cite{of-and-on} and \cite{DS93}. Possibly the best such result is that of Jones, Katz, and Vargas \cite{JKV97}, where they show that for all $A,M,\ve>0$ there is $L>0$ so that if $\Omega$ is {\it any} domain with $B(0,1)\subseteq \Omega\subseteq B(0,M)$ and $\cH^{d}(\d\Omega)\leq A<\infty$, then there is a radial $L(\ve,d,M,A)$-Lipschitz graph $\Gamma$ so that $|\bS^{d}\backslash\{x/|x|:x\in \Gamma\cap \d\Omega\}|<\ve$. 

Another much shorter proof is that of David and Jerison \cite{DJ90}, where they show that the Lipschitz images can also be taken to be boundaries of Lipschitz subdomains of $E^{c}$.  An {\it $L$-Lipschitz domain} is a set of the form
 \[
 T(\{(x,y)\in \bB_{d} \times \bR: f(x)>y> -\sqrt{1-|x|^{2}}\})\]
 where $\bB_{d}$ is the unit ball in $\bR^{d}$, $f:\bR^{d}\rightarrow \bR$ is any nonnegative $L$-Lipschitz supported in $\bB_{d}$, and $T$ is a conformal affine map. Traditionally, Lipschitz domains are defined more generally, but this will suit our purposes. 

\begin{definition}
For $C\geq 2$, an open set $\Omega\subseteq \bR^{d+1}$, $d\geq 1$, is an {\it exterior (or interior) $C$-corkscrew domain} if for all $\xi\in \d\Omega$ and $r\in (0,\diam \d\Omega)$ there is a ball  $B(x,r/C)\subseteq B(\xi,r)\backslash \Omega$ (or a ball $B(x,r/C)\subseteq B(\xi,r)\cap \Omega$). We'll say $\Omega$ is a {\it $C$-corkscrew domain} if it has both exterior and interior corkscrews.
\end{definition}

 \begin{thmii}\label{t:main}
For $d,M,C\geq  1$, there are $\psi=\psi(d,C)>0$ and $L=L(d,C,M)\geq 1$ such that the following holds. Let $\Omega\subseteq \bR^{d+1}$ be a $C$-exterior corkscrew domain, and $B$ a ball of radius $r\in (0,\diam \d\Omega)$ centered on $\d\Omega$ such that $\cH^{d}(B\cap \d\Omega)/r^{d}\leq M<\infty$. Also assume there is $B(x,\rho r/C)\subseteq B\cap \Omega$. Then there is an $L$-Lipschitz domain $\Omega'\subseteq B$ with $\cH^{d}(\d \Omega'\cap \d\Omega)\geq \psi (\rho r)^{d}$. \end{thmii}
\def\TheoremII{Theorem \hyperref[t:main]{II} }

\TheoremII gives an even shorter proof that Semmes surfaces are uniformly rectifiable as follows: Let $E$ be a Semmes surface, $\xi\in E$ and $r>0$, pick an interior corkscrew $B=B(x,r/C)\subseteq B(\xi,r)\backslash E$. Observe that if $\Omega$ is the connected component of $E^{c}$ containing $B$, then $\Omega$ is an exterior corkscrew domain as each ball centered on $E$ must have two corkscrew balls in two different components of $E^{c}$, and so one of them cannot be $\Omega$. Moreover, there is $\xi'\in \d\Omega\cap [x,\xi]$ so that $B\subseteq B(\xi',r)\cap \Omega$ and $\cH^{d}(B(\xi',r)\cap \d\Omega)\leq \cH^{d}(B(\xi',r)\cap E)\leq Ar^{d}$. We can then apply \TheoremII to find a large Lipschitz image in $B(\xi,r)\cap E$. Note that while each component of $E^{c}$ is an exterior corkscrew domain, it may not be interior corkscrew. Just consider $E=\{(x,y):|y|=x^{2},x\in \bR\}$, then the component containing the point $(1,0)$ does not have the $C$-interior corkscrew condition for any $C$.

Badger proves something similar to \TheoremII in \cite[Theorem 2.4]{Bad12}. He observed that the proof in David and Jerison gives a version of \TheoremII if we just assume the boundary is locally $\cH^{d}$-finite rather than Ahlfors regular. His result gives a bit more information, but he needs both interior and exterior corkscrews for his domains.

The additional motivation for finding {\it interior} big pieces of Lipschitz domains aside from uniform rectifiability is a result of Dahlberg \cite{Dah77}, which says that harmonic measures on Lipschitz domains are $A_{\infty}$-weights. Using a version of \TheoremII, David and Jerison showed harmonic measure is an $A_{\infty}$-weight if $\Omega$ has Ahlfors regular boundary and is a nontangentially accessible (or NTA) domain, which happen to be connected corkscrew domains. (We will not discuss the definition of an NTA domain and refer the reader to its inception in \cite{JK82}.) Badger in turn, using his version of \TheoremII, shows that $\cH^{d}|_{\d\Omega}\ll \omega$ ($\omega$ denoting harmonic measure) if we only assume $\cH^{d}|_{\d\Omega}$ is Radon.\\

The common thread in our proofs of \TheoremI and \TheoremII is to use Fubini's theorem to show that if a portion of $\d\Omega\cap B$, say, has large projections in several directions, then there must be a set $E$ of points in $\d\Omega\cap B$ of large measure which are "visible" from a positive measure set of directions (that is, for each $\xi\in E$ there are line segments emanating from $\xi$ in many directions without hitting $\d\Omega$) . We then show that, for each $\xi\in E$, the set of directions are dense enough around one particular direction that in fact there is a whole spherical cap of directions that $\xi$ is visible from (since if one of those rays did hit the boundary, we would find an exterior corkscrew that would have to block one of these directions). Thus, $\xi$ is the apex of a cone contained in $\Omega$. From here it is not too hard to show that a large portion of $E$ lies in the boundary of a Lipschitz domain. \\

The author would like to thank Mihalis Mourgoglou for his helpful discussions. Part of this work was done while the author was attending the 2015 Research Term on Analysis and Geometry in Metric Spaces at the ICMAT. 

\section{Preliminaries}

We write $B(x,r)$ for the closed ball in $\bR^{d+1}$ centered at $x$ of radius $r$ and $B_{\bS^{d}}(\theta,r)$ denote the closed ball in $\bS^{d}$ centered at $\theta\in \bS^{d}$ of radius $r$ with respect to arclength. In particular, for $\delta>0$, we let $B(\delta)=B_{\bS^{d}}(-e_{d+1},\delta)$. For a set $A$, we will let $\cH^{d}(A)$ and $|A|$ denote the $d$-dimensional Hausdorff measure (whose definition can be found in \cite{Mattila}) normalized so that $w_{d}:=|B(0,1)\cap \bR^{d}|$ is equal to the $d$-dimensional Lebesgue measure of $B(0,1)\cap \bR^{d}$. For $x\in \bR^{d+1}$, we set $\dist(x,A)=\inf\{|x-y|:y\in A\}$.  

In this section, we prove three lemmas that will be used in the proofs of \TheoremI and \TheoremII.

\begin{lemma}\label{l:porous}
There is $\delta_{0}=\delta_{0}(d)>0$ so that for all $\eta,\kappa>0$ and $\delta\in (0,\delta_{0})$, there is $c_{0}=c_{0}(\kappa,\eta,d)>0$ such that for any $ A\subseteq B(\delta)$ with $|A|\geq \kappa |B(\delta)|$, there is $\theta_{A}\in {A}$ and $r_{A}\in( c_{0}\delta,\delta)$ so that any $\theta \in B_{\bS^{d}}(\theta_{A},r_{A})$ is at most $\eta r_{A}$ from ${A}$ with respect to the arclength metric on $\bS^{d}$.
\end{lemma}

\begin{proof}
Let $Q_{0}=[-\delta,\delta]^{d}\subseteq \bR^{d}$. Here, when we say {\it dyadic cube}, we mean a set $Q$ of the form $\prod_{i=1}^{d}[j_{i}2^{k},(j_{i}+1)2^{k}]$ for any integers $j_{1},...,j_{d},k\in \bN$ and we will denote the sidelength of $Q$ by $\ell(Q)$.  Let $\pi$ be the orthogonal projection onto $\bR^{d}$ and $A'=\pi(A)\subseteq Q_{0}\subseteq \bR^{d}$. For $\delta$ small enough, we can guarantee $\pi:\bS^{d}\cap \pi^{-1}(Q_{0})\rightarrow Q_{0}$ has a $2$-bi-Lipschitz inverse on $Q_{0}$ (with respect to the arclength metric in $\bS^{d}$), and so $| {A'}|\geq 2^{-d}| {A}|\geq 2^{-d}\kappa|B(\delta)|\geq c\kappa \delta^{d}$ for some $c=c(d)$.

Let $\{Q_{j}\}$ be the maximal dyadic cubes in $Q_{0}\backslash {A'}$. For $Q\subseteq Q_{0}$ define
\[\lambda(Q)=\sum_{Q_{j}\subseteq Q}\frac{\ell(Q_{j})^{d+1}}{\ell(Q)^{d+1}}\]
where the sum is zero if $Q$ contains no $Q_{j}$. Then
\begin{align}
\sum_{Q\subseteq Q_{0}} \lambda(Q)|Q| 
& =\sum_{Q\subseteq Q_{0}}\sum_{Q_{j}\subseteq Q} |Q_{j}| \frac{\ell(Q_{j})}{\ell(Q)}
=\sum_{j} |Q_{j}|\sum_{Q_{j}\subseteq Q\subseteq Q_{0}}\frac{\ell(Q_{j})}{\ell(Q)} \notag \\
& \leq 2\sum_{j}|Q_{j}|\leq 2|Q_{0}|.
\label{e:lambda}
\end{align}

Note that there are at least $M_{n}:=2^{nd-1} | {A'}|/|Q_{0}| \geq 2^{nd-1}c\kappa$ dyadic cubes of sidelength $2^{-n}\ell(Q_{0})$ that intersect ${A'}$. Suppose there is $N\in \bN$ such that all cubes $Q$ intersecting ${A'}$ of sidelength at least $2^{-N}\ell(Q_{0})$ contain a $Q_{j}$ with $\ell(Q_{j})\geq \eta \ell(Q) $ (so $\lambda(Q)\geq \eta^{d+1}$). Then
\begin{equation}\label{e:lambda2}
\eta^{d+1}Nc\kappa 2^{-1}|Q_{0}|
\leq \sum_{n=0}^{N-1}\eta^{d+1}M_{n}2^{-nd}|Q_{0}|
\leq \sum_{Q\subseteq Q_{0} : Q\cap A'\neq\emptyset}\lambda(Q)|Q|.
\end{equation}
Then \eqn{lambda} and \eqn{lambda2} imply $N\leq N_{0}:=\frac{4}{\eta^{d+1}c\kappa}$. Hence, there is $Q$ with $\ell(Q)\geq 2^{-N_{0}}\ell(Q_{0})$ that intersects ${A'}$ and so that there is no dyadic cube in $Q\backslash  {A'}$ of sidelength at least $\eta\ell(Q)$.  In particular, every point in $Q$ is at most $\eta\sqrt{d}\ell(Q)$ away from a point in ${A'}$. Thus every point in $\pi^{-1}(Q)\cap \bS^{d}$ is at most $2\sqrt{d}\eta\ell(Q)$ in the path metric on $\bS^{d}$ from $ {A}$. Thus, we can find a point $\theta_{A}\in  {A}\cap \pi^{-1}(Q)$ (say, the point in $A$ whose projection is closest to the center of $Q$) and $r_{A}>0$ so that $B_{\bS^{d}}(\theta_{A},r_{A})\subseteq \pi^{-1}(Q)\cap \bS^{d}$, $c_{1}\ell(Q)\leq r_{A}< \ell(Q)\leq\delta$ for some $c_{1}=c_{1}(d)>0$, and every point in $B_{\bS^{d}}(\theta_{A},r_{A})$ is at most $c_{2}\eta r_{A}$ from $ {A}$ where $c_{2}$ depends only on $d$. 
\end{proof}

\begin{lemma}\label{l:cones}
Let $\delta\in (0,\delta_{0})$, $\kappa>0$, $0<4C\eta<\upsilon<1/4$ be small, $\theta_{0}\in \bS^{d+1}$ and $A\subseteq B_{\bS^{d}}(\theta_{0},\delta)$ be such that $|A|\geq \kappa |B_{\bS^{d}}(\theta_{0},\delta)|$. Suppose $\theta_{A}\in A$ and $r_{A}>0$ are such that for all $\theta\in B_{\bS^{d}}(\theta_{A},r_{A})$, $\theta$ is at most $\eta r_{A}$ from $A$. Let $\Omega$ be a $C$-exterior corkscrew domain, $\xi\in \d\Omega$, $t\in (0,\diam \d\Omega)$, and assume $(\xi,\xi+\theta t)\subseteq \Omega$ for all $\theta\in A$. Let 
\[
C(\xi,\theta_{A},\upsilon r_{A}, t)
 = B(\xi,r_{A})\cap \{x\in \bR^{d+1}: (x-\xi)/|x-\xi|\in B_{\bS^{d}}(\theta_{A},\upsilon r_{A})\}.
\]
Then $C(\xi,\theta_{A},\upsilon r_{A},t/2)\subseteq \Omega$.
\end{lemma}

\begin{proof}

Without loss of generality, we will assume $\theta_{0}=-e_{d+1}$ so that $B_{\bS^{d}}(\theta_{0},\delta)=B(\delta)$, and that $\xi=0$. Set
\[
C_{1}= C(0,\theta_{A},\upsilon r_{A},t/2), \;\; C_{2}= C(0,\theta_{A},2\upsilon r_{A},t).\]
Suppose $C_{1}\cap \Omega^{c}\neq\emptyset$. Since $(0,\frac{t}{2}\theta_{A})\subseteq C_{1} \cap \Omega$, connectivity of this cone implies there is $\zeta\in C_{1}\cap \d\Omega$. We claim that 
\begin{equation}\label{e:Bzeta}
B(\zeta,|\zeta|\sin \upsilon r_{A})\subseteq C_{2}.
\end{equation}
Let $w\in B(\zeta, |\zeta|\sin \upsilon r_{A} )$. The largest angle $w$ may have with $\zeta$ is if $[0,w]$ is tangent to the ball $ B(\zeta, |\zeta|\sin \upsilon r_{A})$, in which case the angle is $\upsilon r_{A}$. As the angle of $\zeta$ with $\theta_{A}$ is at most $\upsilon r_{A}$, the angle of $w$ with $\theta_{A}$ is at most $2\upsilon r_{A}$. Thus, $w/|w|\in B_{\bS^{d}}(\theta_{A},2\upsilon r_{A})$, and moreover, for all $w\in  B(\zeta, |\zeta|\sin \upsilon r_{A} )$
\[
|w|\leq |\zeta|(1+\sin \upsilon r_{A})\leq  \frac{t}{2} (1+1)\leq t\]
  (see Figure \ref{wz}.a). These two facts imply \eqn{Bzeta}. 
%
\begin{figure}
\includegraphics[width=205pt]{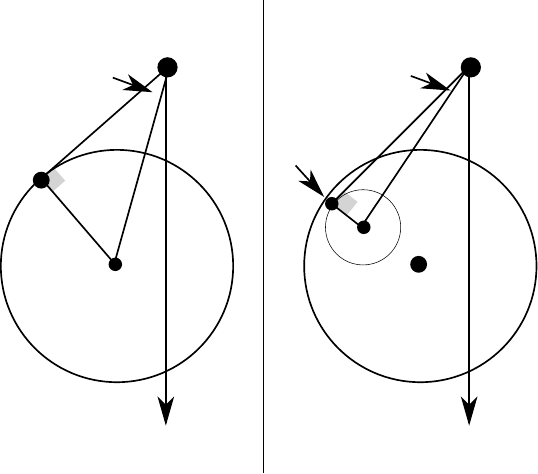}
\begin{picture}(0,0)(212,0)
\put(40,65){$\zeta$}
\put(155,65){$\zeta$}
\put(60,160){$0$}
\put(175,160){$0$}
\put(10,120){$w$}
\put(133,90){$z$}
\put(110,120){$y$}
\put(155,155){$\psi$}
\put(135,70){$B_{z}$}
\put(-15,86){$|\zeta|\sin \upsilon r_{A}$}
\put(70,20){$\theta_{A}$}
\put(187,20){$\theta_{A}$}
\put(30,150){${\upsilon r_{A}}$}
\put(40,0){a.}
\put(150,0){b.}
\end{picture}
\caption{
}
\label{wz}
\end{figure}

Since $\Omega$ has $C$-exterior corkscrews, we may find 
\begin{equation}\label{e:bz}
B_{z}:=B\ps{z, \frac{|\zeta|\sin \upsilon r_{A}}{C}}\subseteq B(\zeta,|\zeta|\sin \upsilon r_{A})\backslash \Omega.\end{equation}
 If $\psi>0$ is such that 
\[ B_{\bS^{d}}\ps{{z}/{|z|},\psi}=\ck{{y}/{|y|}:y\in B_{z}},\]  then
\begin{equation}\label{e:z/z}
B_{\bS^{d}}(z/|z|,\psi)\subseteq B_{\bS^{d}}(\theta_{A},r_{A})\backslash A.
\end{equation}
Indeed, \eqn{Bzeta}, \eqn{bz}, and the fact that $\upsilon<1/2$  imply 
\[
B_{\bS^{d}}(z/|z|,\psi)\subseteq B_{\bS^{d}}(\theta_{A},2 \upsilon r_{A})\subseteq B_{\bS^{d}}(\theta_{A},r_{A}).\]
Since $(0,y)\cap \Omega^{c}\supseteq (0,y)\cap B_{z}\neq\emptyset$ for all $y\in B_{z}$, we must have $B_{\bS^{d}}(z/|z|,\psi) \subseteq A^{c}$, and this completes the proof of \eqn{z/z}. 

Since 
\[
|z|\leq |\zeta|+\frac{|\zeta|\sin \upsilon r_{A}}{C}\leq 2|\zeta|\]
 and $\sin \upsilon r_{A}\geq \upsilon r_{A}/2$ for $\upsilon$ small enough, if $[0,y]$ is tangent to $B_{z}$ (see Figure \ref{wz}.b), then by our assumption on $\eta$,
\[\psi\geq\sin \psi = \frac{C^{-1}|\zeta|\sin \upsilon r_{A}}{|z| }\geq \frac{\upsilon r_{A}}{4C}>\eta r_{A}.\]
Thus, we know $A\cap B_{\bS^{d}}(z/|z|,\psi)\neq\emptyset$ by assumption, contradicting \eqn{z/z}.
\end{proof}

\begin{lemma}\label{l:T}
Let $\Sigma\subseteq \bR^{d+1}$ be closed, $B$ be a ball centered on $\Sigma$ of radius $r>0$, and suppose $B(0,\frac{r}{2C})$ and $B(ae_{d+1},r/C)$ are contained in different components of $\Sigma^{c}$ in $B$ for some $a>0$. There is $\delta_{1}=\delta_{1}(C,a)>0$ so that the following holds. For $\delta\in(0,\delta_{1})$  and $\theta\in \bS^{d}$, let 
\[
L_{\theta}=\{x\in\bR^{d+1}:x\cdot \theta=0\}, \;\; D_{\theta}=B\ps{0,\frac{r}{2C}}\cap L_{\theta},\]
 $\pi_{\theta}$ be the orthogonal projection onto $L_{\theta}$, $T$ be the convex hull of $B(0,\frac{r}{2C})\cup B(ae_{d+1},r/C)$ and $S=T\cap \Sigma$. Then $\pi_{\theta}(S)\supseteq D_{\theta}$ for all $\theta\in B(\delta)$. 
\end{lemma}

\begin{proof} 

Note that $B(ae_{d+1},\frac{ r}{2C})$ contained in the interior of $B(ae_{d+1},r/C)$, and there is $\delta_{1}>0$ so that if $\Theta$ is a rotation about zero in any direction by angle $\theta\in B(\delta_{1})$, then we still have $\Theta(B(ae_{d+1},\frac{ r}{2C}))\subseteq B(ae_{d+1},r/C)$, and so $\pi_{\theta}(S)\supseteq \pi_{\theta}(\Theta(B(ae_{d+1},\frac{ r}{2C})))=D_{\theta}$.
\end{proof}

%
%

\section{Proof of \TheoremI}

\begin{definition}\label{d:tangent}
For a set $\Sigma\subseteq \bR^{n}$, the {\it contingent of $\Sigma$ at $\xi\in \Sigma$} is the union of all half-lines $\{\theta t:t\geq 0\}$  for which there is $\xi_{i}\in \Sigma\backslash\{\xi\}$ converging to $\xi$ so that $(\xi_{i}-\xi)/|\xi_{i}-\xi|\rightarrow \theta$. We say that $\Sigma$ has a {\it $d$-dimensional tangent} at $\xi\in \Sigma$ if the contingent is a $d$-dimensional plane.
\end{definition}

\begin{lemma}\label{l:contingent}
Given a set $\Sigma\subseteq \bR^{d+1}$, let $P$ be the set of points in $\Sigma$ where the contingent is not all of $\bR^{d+1}$. Then $P$ has $\sigma$-finite $\cH^{d}$-measure and for $\cH^{d}$-almost every $\xi\in P$, the union of half-lines in the contingent is either a $d$-plane (in which case $\xi$ is a tangent point for $\Sigma$) or a half-space. 
\end{lemma}

The planar case of this lemma is stated in \cite[p. 266]{Saks-Integral}, but as mentioned in \cite{CW15} after Lemma 6, the above version is proved similarly. 

\begin{remark}
The definition we use for a tangent above is the same as the one given in \cite[Definition 5]{CW15}, except that in their definition, $\Sigma$ is always homeomorphic to a cube, while in our definition we allow $\Sigma$ to be any set. There are many different definitions of tangents in the literature, c.f. \cite[p. 60]{GM08}, \cite{AMT16}, and \cite{Federer}, where the latter two are also defined for general sets. The common thread to each of the definitions is that, for a point $\xi$ to be a tangent for $\Sigma$, $\Sigma$ should look flatter and flatter in smaller and smaller balls around $\xi$. In \cite{AMT16}, for example, we say $\Sigma$ has a { $d$-dimensional tangent} at $\xi\in \Sigma$ if there is a $d$-dimensional plane $V$ containing $\xi$ (not necessarily unique) so that $\lim_{r\rightarrow 0} \sup_{\zeta\in B(\xi,r)\cap \Sigma}\dist(\zeta,V)/r= 0$. This is similar to the definition given in \cite{Federer}. Definition \ref{d:tangent} is stronger in the sense that, if a point has a $d$-dimensional tangent with respect to our definition, it also has one with respect to these other definitions and the tangent plane is unique.

%
\end{remark}

We now begin the proof of \TheoremI. Let $\Sigma$ satisfy the $2$-ball condition with constant $C\geq 2$ and $B$ be a ball centered on $\Sigma$. By scaling, we may assume $B$ has radius $1$. We will show that, for each point in a subset of $\Sigma\cap B$ of positive $\cH^{d}$-measure, the contingent is not $\bR^{d+1}$ or a halfspace, so that by \Lemma{contingent} almost all of these points will be tangent points.

Since $\Sigma$ has the 2-ball condition, we may find two balls of radius $1/C$ in two different components of $ \Sigma^{c}$ in $B$ of radius $1/C$. By rotation and translation, we may assume one is $B(0,1/C)$ and the other $B(ae_{d+1},1/C)$ where $2-2/C\geq a\geq 2/C$ . Let $S$ and $D_{\theta}$ be the sets from \Lemma{T} with $0<\delta<\min\{\delta_{0},\delta_{1}(a,C)\}$. Let $A_{i}$ be a countable partition of $S$ into sets of finite $\cH^{d}$-measure. Let $t_{i}^{-1}=|A_{i}|2^{i}$ and set $\mu=\sum_{i=1}^{\infty} t_{i} \cH^{d}|_{A_{i}}$, so $\mu$ is a finite Borel measure with support equal to $S$.  Set
\[ h(\theta,\xi)=\inf\{|\xi-\zeta|:\zeta\in S\cap \pi_{\theta}^{-1}(\pi_{\theta}(\xi))\backslash\{\xi\}\}\]
with the convention that $\inf \emptyset=\infty$.
\begin{lemma} The function $h$ is a Borel function on $B(\delta)\times S$.
\label{l:hBorel}
\end{lemma}
We postpone the proof for now until the end of the section. If we set
\[F_{t}:=\{(\theta,\xi)\in B(\delta)\times S: t<h(\theta,\xi)\}\]
then $F_{t}$ is Borel for all $t\geq 0$. Note that $F_{0}$ is the set of pairs $(\theta,\xi)$ so that $\xi\in S$ is an isolated point in $\pi_{\theta}^{-1}(\pi_{\theta}(\xi))\cap S$. Set 
\[ S_{t}(\theta)=\{\xi\in S: (\theta,\xi)\in F_{t}\}\mbox{ and }\theta_{t}(\xi)=\{\theta\in B(\delta): (\theta,\xi)\in F_{t}\}.\]
By Theorem 10.10 in \cite{Mattila}, for each $i\in \bN$, $\cH^{0}(A_{i}\cap \pi_{\theta}^{-1}(x))<\infty$ for every $\theta\in B(\delta)$ and for almost every $x\in L_{\theta}$. Thus, $S\cap \pi_{\theta}^{-1}(x)$ is countable for every $\theta\in \bS^{d}$ and almost every $x\in L_{\theta}$, and so it must contain an isolated point if it is nonempty. By  \Lemma{T}, $\pi_{\theta}(S)\supseteq D_{\theta}$ and so $S\cap \pi_{\theta}^{-1}(x)\neq\emptyset$ for each $x\in D_{\theta}$ and $\theta\in B(\delta)$, thus for each $\theta\in B(\delta)$ and almost every $x\in D_{\theta}$, $S\cap \pi_{\theta}^{-1}(x)\neq\emptyset$ must have an isolated point, or in other words, $S_{0}(\theta)\cap \pi_{\theta}^{-1}(x)\neq\emptyset$. Hence $|S_{0}(\theta)|\geq |\pi_{\theta}(S)|\geq |D_{\theta}|>0$, thus $\mu(S_{0}(\theta))>0$ for all $\theta\in B(\delta)$. Since $F_{0}$ is Borel, we may integrate, apply Fubini, and use the monotone convergence theorem to get
\[0<\int_{B(\delta)} \mu(S_{0}(\theta))d\theta= \int_{S} |\theta_{0}(\xi)|d\mu(\xi)
=\lim_{t\rightarrow 0} \int_{S} |\theta_{t}(\xi)|d\mu
.\]
Thus, if we set $E_{t,s}=\{\xi\in S: |\theta_{t}(\xi)|>s\}$, then $|E_{t,s}|>0$ for some $t\in (0,\diam\Sigma)$ and $s>0$. Let $\{\Omega_{i}\}_{i\in I}$ be the components of $\Sigma$ and for $\xi\in E_{t,s}$ set
\[\theta_{ij}(\xi)=\{\theta\in \theta_{t}(\xi): (\xi,\xi+t\theta]\subseteq \Omega_{i},\;\;  (\xi,\xi-t\theta]\subseteq \Omega_{j}\}.\]
Then since $\Sigma$ is closed 
\begin{align*}
\bigcup_{i,j} \theta_{ij}(\xi)
& =\ck{\theta\in \theta_{t}(\xi): (\xi, \xi+t\theta]\cup (\xi, \xi-t\theta]\subseteq \Sigma^{c}}\\
& =\ck{\theta\in \theta_{t}(\xi): t<h(\theta,\xi)}
=\theta_{t}(\xi)
\end{align*}
 and so $|\theta_{ij}(\xi)|>0$ for some $i,j\in I$. As observed in the introduction, each $\Omega_{i}$ is a $C$-exterior corkscrew domain since $\Sigma$ has the $2$-ball condition. Pick $0<4C\eta<\upsilon<1/4$ as in \Lemma{cones} and apply \Lemma{porous} to $A=\theta_{ij}(\xi)$ to get $\theta_{\theta_{ij}(\xi)}\in \theta_{ij}(\xi)$ and $r_{\theta_{ij}(\xi)}>0$ (depending on $\eta,d$, and $\kappa=|A|/|B(\delta)|\geq s/|B(\delta)|>0$). Since $(\xi,\xi+t\theta)\subseteq \Omega_{i}$ for each $\theta\in \theta_{ij}(\xi)$, by \Lemma{cones} with $\Omega=\Omega_{i}$, $B_{\bS^{d}}(\theta_{0},\delta)=B(\delta)$, we have 
\[C(\xi,\theta_{\theta_{ij}(\xi)},\upsilon r_{\theta_{ij}(\xi)},t/2)^{}\subseteq \Omega_{i} \subseteq \Sigma^{c}.\] 
By applying \Lemma{porous} with $\Omega=\Omega_{j}$, $A=-\theta_{ij}(\xi)$ and $B_{\bS^{d}}(\theta_{0},\delta)=-B(\delta)$, we also get that $C(\xi,-\theta_{\theta_{ij}(\xi)},\upsilon r_{\theta_{ij}(\xi)},t/2)\subseteq \Sigma^{c}$. Thus, the contingent of $\Sigma$ at $\xi$ does not contain any half-line from $\xi$ passing through $B_{\bS^{d}}(\theta_{\theta_{ij}(\xi)},\upsilon r_{\theta_{ij}(\xi)})\cup B_{\bS^{d}}(-\theta_{\theta_{ij}(\xi)},\upsilon r_{\theta_{ij}(\xi)})$, thus the contingent cannot be $\bR^{d+1}$ or a half-space. Since this holds for each $\xi\in E_{t,s}$, by \Lemma{contingent}, we conclude that $\Sigma$ has tangents at almost every point in $E_{t,s}\subseteq B\cap \Sigma$. Since $|E_{t,s}|>0$, we are done.

\begin{proof}[Proof of \Lemma{hBorel}]
For $\ve>0$ let 
\[
L_{\theta,\xi,\ve}=[\xi+\ve \theta,\xi+\ve^{-1}\theta]\cup [\xi-\ve\theta,\xi-\ve^{-1}\theta]\]
and
\[ h_{\ve}(\theta,\xi)=\dist(\xi, S\cap L_{\theta,\xi,\ve}).\]
Note that $h_{\ve}$ decreases pointwise on $\bS^{d}\times\d\Omega$ to $h$ as $\ve\downarrow 0$, and thus it suffices to show that each $h_{\ve}$ is Borel measurable for each $\ve>0$. In fact, we will show $h_{\ve}$ is lower semicontinuous.

Let $(\theta_{j},\xi_{j})\rightarrow (\theta,\xi)\in \bS^{d}\times S$, we will show $h_{\ve}(\theta,\xi)\leq\liminf h_{\ve}(\theta_{j},\xi_{j})$  We can clearly assume $\liminf h_{\ve}(\theta_{j},\xi_{j})<\infty$. By passing to a subsequence if necessary, we may also assume $h_{\ve}(\theta_{j},\xi_{j})$ converges. Let $\zeta_{j}\in S\cap L_{\theta_{j},\xi_{j},\ve}$ be so that $h_{\ve}(\theta_{j},\xi_{j})=|\xi_{j}-\zeta_{j}|$. Passing to another subsequence, we may assume $\zeta_{j}\rightarrow \zeta\in S \cap L_{\theta,\xi,\ve}$ (since $\zeta_{j}$ is a bounded sequence in $S$ and $S$ is closed---this is why we have defined our balls to be closed). Then by definition of $h_{\ve}$,
\[h_{\ve}(\theta,\xi)\leq |\xi-\zeta|=\lim |\xi_{j}-\zeta_{j}|=\lim h_{\ve}(\theta_{j},\xi_{j}).\]
\end{proof}

\section{Proof of \TheoremII}

Let $\Omega$ be a $C$-exterior corkscrew domain, $B$ be a ball centered on $\d\Omega$. We will first prove \TheoremII assuming $\rho=1$, so we assume there is a ball $B(x,r/C)\subseteq B\cap \Omega$. Without loss of generality, we may assume $B(ae_{d+1},r/C)\subseteq B\backslash \Omega$ and $B(0,\rho r/C)\subseteq B\cap \Omega$.  

 Let $S$, $D$, and $D_{\theta}$ be as in \Lemma{T} for $\Sigma=\d\Omega$ and $\delta<\min\{\delta_{0},\delta_{1}(a,C),\delta_{2}\}$ where $\delta_{2}>0$ is a number yet to be determined. Define (see Figure 2.a)
\[
G=\{(\theta,\xi)\in B(\delta)\times S:(\pi_{\theta}(\xi),\xi)\subseteq \Omega.\}, \;\; S(\theta)=\{\xi\in S:(\theta,\xi)\in G\},\]
\[\theta(\xi)=\{\theta\in B(\delta): (\theta,\xi)\in G\}\mbox{ and } E_{\kappa}=\{\xi\in S: |\theta(\xi)|\geq \kappa |B(\delta)|\}\]
where 
\begin{equation}\label{e:kappa}
\kappa=\frac{w_{d}}{2^{d+1}C^{d}M}\leq \frac{|D|}{2|S|}.
\end{equation}

\begin{lemma} \label{l:GBorel}
There is $\delta_{2}=\delta_{2}(d)>0$ so that for $0<\delta<\delta_{2}$, $G$ is Borel.
\end{lemma}

We postpone the proof of this to the end of the section and assume further that $\delta<\delta_{2}$. Note that $\pi_{\theta}(S(\theta))\supseteq D_{\theta}$, hence $|D_{\theta}|\leq |S(\theta)|$ and since $G$ is Borel, we may integrate
\begin{align*}
|B(\delta)| |D|
& \leq \int_{B(\delta)}|S(\theta)|d\theta
=\int_{S}|\theta(\xi)|d\cH^{d}(\xi)\\
& \leq \int_{|\theta(\xi)|\leq \kappa}\kappa |B(\delta)| d\cH^{d}(\xi)+ \int_{|\theta(\xi)|> \kappa} |B(\delta)| d\cH^{d}(\xi)\\
& \leq \kappa |B(\delta)||S|+|B(\delta)| |E_{\kappa}|
\stackrel{\eqn{kappa}}{\leq} |D||B(\delta)|/2 + |B(\delta)||E_{\kappa}|
\end{align*}

\begin{figure}[h]
\includegraphics[width=120pt]{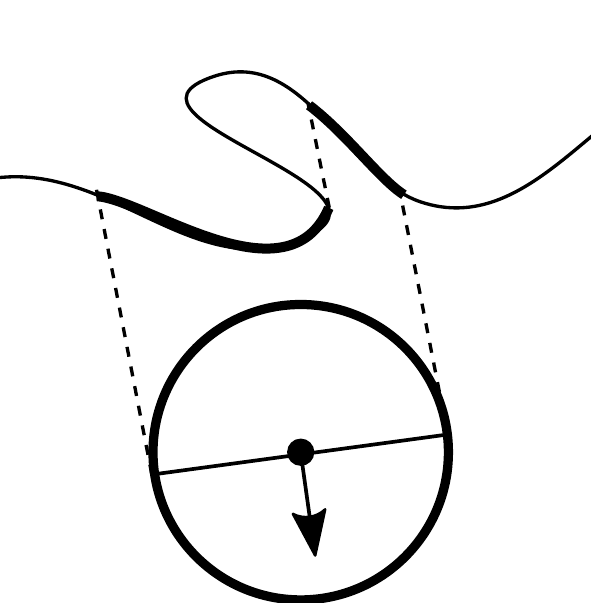}
\includegraphics[width=230pt]{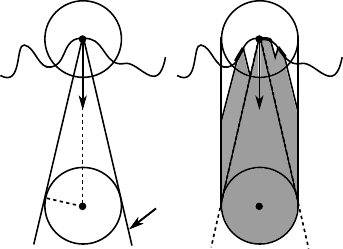}
\begin{picture}(0,0)(350,0)
\put(175,150){$B_{j}$}
\put(175,15){$B_{ij}$}
\put(205,30){$C_{i}(\xi)$}
\put(285,70){$\Omega_{ij}$}
\put(173,90){$x_{i}$}
\put(60,35){$D_{\theta}$}
\put(45,10){$\theta$}
\put(25,80){$S(\theta)$}
\put(65,100){$S(\theta)$}
\put(95,70){$\d\Omega$}
\put(210,135){$\d\Omega$}
\put(335,135){$\d\Omega$}
\put(155,35){$\tau$}
\put(50,-10){a.}
\put(170,-10){b.}
\put(285,-10){c.}
\end{picture}
\label{omega}
\caption{}
\end{figure}
which implies $|E_{\kappa}|\geq |D|/2= w_{d}\ps{\frac{r}{2C}}^{d}/2$.

Let $\upsilon,\eta$ and $c_{0}$ be as in Lemmas \ref{l:porous} and \ref{l:cones}, and let $\xi\in E_{\kappa}$, let $\theta_{\theta(\xi)}$ and $r_{\theta(\xi)}$ be $\theta_{A}$ and $r_{A}$ from \Lemma{porous} with $A=\theta(\xi)$. Note that since $|\theta(\xi)|\geq \kappa |B(\delta)|$, we know $r_{\theta(\xi)}\geq c_{0}\delta$ where $c_{0}$ depends only on $d$ and $\kappa$ (and so just on $d$, $C$, and $M$). Also, if $\theta\in \theta(\xi)$, then $(\pi_{\theta}(\xi),\xi)\subseteq \Omega$, and since $\xi\in B(0, r/C)^{c}$ while $\pi_{\theta}(\xi)\in D_{\theta}\subseteq B(0,\frac{r}{2C})$, we have $|\pi_{\theta}(\xi)-\xi|\geq \frac{r}{2C}$. Hence, if $t=\frac{r}{2C}$, then $(\xi,\xi+t\theta)\subseteq \Omega$ for each $\theta\in \theta(\xi)$.  Pick a maximally $\upsilon c_{0}\delta/2$-separated set $\{x_{i}\}_{i=1}^{n_{1}}\subseteq \bS^{d}$ (with respect to the arclength metric), so that for all $\xi\in E$, there is $x_{i} \in B_{\bS^{d}}(\theta_{\theta(\xi)},\upsilon c_{0}\delta/2)$, and so by \Lemma{cones} with our choice of $t$,

\begin{align*}
C_{i}(\xi)& :=C(\xi,x_{i},\upsilon c_{0}\delta/2,t/2)\subseteq C(\xi,\theta_{\theta(\xi)},\upsilon c_{0} \delta,t/2) \\
& \subseteq C(\xi,\theta_{\theta(\xi)},\upsilon r_{\theta(\xi)},t/2) \subseteq \Omega
\end{align*}
Moreover, as $\xi\in B$ and $t/2=\frac{r}{4C}$, $C_{i}(\xi)\subseteq (1+\frac{1}{4C})B$ as well. Let $\tau=\frac{t}{4}\sin (\upsilon c_{0}\delta/2)$ and $\{y_{j}\}_{j=1}^{n_{2}}$ be a maximally $\tau$-separated set in $E_{\kappa}$. Set $B_{j}=B(y_{j},\tau) $ and
\[E_{ij}=\{\xi\in E_{\kappa}\cap B_{j},: x_{i}\in B_{\bS^{d}}(\theta_{\xi},\upsilon c_{0}\delta/2)\}.\]
Then $B_{ij}:=B_{j}+tx_{i}/4\subseteq C_{i}(y_{j})$ (see Figure 2.b). We now set
 \begin{equation}\label{e:omegaij}
 \Omega_{ij}=\ps{\bigcup_{\xi\in E_{ij}} C_{i}(\xi) \cap \co(B_{j}\cup B_{ij})}^{\circ}\subseteq \Omega\cap \ps{1+\frac{1}{4C}}B
 \end{equation}
 where $\co$ denotes the convex hull, see Figure 2.c.
 
 The above is an $L$-Lipschitz domain with $L=\sec (\upsilon c_{0}\delta/2)$ such that $E_{ij}\subseteq \d\Omega_{ij}\cap \d\Omega$ (see for example Lemma 15.13 of \cite{Mattila}). Moreover, we can find $i,j$ so that 
 \[|E_{ij}|\geq \frac{|E_{\kappa}|}{ n_{1}n_{2}}\geq \frac{w_{d} r^{d}}{n_{1}n_{2}2^{d+1}C^{d}}.\]
 Since $n_{1}$ and $n_{2}$ are bounded above by a number depending only on $d$ and the number $\upsilon c_{0}\delta/2$, we have that $|E_{ij}|\geq  cr^{d}$ for some $c=c(d,C)$. Then $\Omega'=\Omega_{ij}$ is our desired domain and we are done.

Now we consider general $\rho\leq 1$, so assume there is $B(x,\rho r/C)\subseteq B\cap \Omega$. We can assume that, of all balls of the same radius contained in $B\cap \Omega$, $x$ is closest to the center. In this way, if $H$ is the half-sphere of $\d B(x,\rho r/C)$ with pole facing the center of $B$, we can assume there is $y\in H\cap \d\Omega$ (otherwise, we could move $B(x,\rho r/C)$ closer to the center). Then it is not hard to show that $\dist(y, B^{c})\geq c_{d} \rho r$ for some $c_{d}>0$. Then since $y\in \d B(x,\rho r/C)$ and $B(x,\rho r/C)\subseteq \Omega$, we have that $B'\subseteq B$ contains a ball of half its radius (and in particular, at least $1/C$ times its radius) that is also contained in $\Omega$. We now apply our work in the $\rho=1$ case of \TheoremII to $B'$ (and recalling \eqn{omegaij}) to get the desired Lipschitz domain contained in 
\[
\ps{1+\frac{1}{4C}}B' \subseteq 2B' = B(y,c_{d} \rho r) \subseteq B\]
and this finishes the proof.

\begin{proof}[Proof of \Lemma{GBorel}]
A similar argument appears in Remark 2.2 of \cite{JJMMT03}, though not in this generality. 

For $\delta>0$ small enough, there is $\Theta:B(\delta)\rightarrow O(d+1)$ a continuous map that is a homeomorphism onto its image in the orthogonal group such that $\Theta(\theta)(e_{d+1})=\theta$ for all $\theta\in \bS^{d}$. One way to find this map is as follows: The function $h:O(d+1)\rightarrow \bS^{d+1}$ defined by $h(\Theta)=\Theta(e_{d+1})$ is a differentiable map and if $X$ are the set of critical points, then $|h(X)|=0$ by Sard's theorem. For all $\Theta\in O(d+1)$, $h=\Theta^{-1}\circ h\circ \Theta$, so by symmetry of the sphere and $O(d+1)$ we know $h(X)=\emptyset$, thus $X=\emptyset$, and hence $h$ has full rank everywhere. By the inverse function theorem, for $\delta>0$ small enough we can find a $d$-surface $S$ containing the identity map $I\in O(d+1)$ so that $\Theta:=h^{-1}:B(\delta)\rightarrow S$ is a homeomorphism. 

For $\theta\in B(\delta)$, let $\Omega_{\theta}=\Theta(\theta)^{-1}(\Omega)$ and $S_{\theta}=\Theta(\theta)^{-1}(S)$.
For $x\in D_{\theta}$, let $\xi(\theta,x)\in S$ be the unique point such that $\pi_{\theta}(\xi(\theta,x))=x$ and $(\xi(\theta,x),x)\subseteq \Omega$. For $x\in D$, let $\xi'(\theta,x)=\Theta(\theta)^{-1}(\xi(\theta,\Theta(\theta)(x)))$, so this is the unique point in $S_{\theta}$ so that $\pi(\xi'(\theta,x))=x$ and $(\xi'(\theta,x),x)\subseteq \Omega_{\theta}$. Now define $g(\theta,x)=|\xi'(\theta,x)-x|$.

We claim $g:B(\delta)\times D\rightarrow \bR$ is lower semicontinuous. Let $(\theta_{j},x_{j})\in B(\delta)\times D$ converge to $(\theta,x)\in B(\delta)\times D$ we need to show $g(\theta,x)\leq \liminf g(\theta_{j},x_{j})$. By passing to a subsequence, we can assume $g(\theta_{j},x_{j})$ converges, and also that $\xi'(\theta_{j},x_{j})$ converges to a point $\zeta\in S_{\theta}$. Then $\pi(\zeta)=x$ and by definition of the function $\xi'$, we must have 
\[ g(\theta,x)\leq |\zeta-x|=\lim |\xi'(\theta_{j},x_{j})-x_{j}|=\lim g(\theta_{j},x_{j}),\]
and this proves the claim.

The set $\Gamma=\{(\theta,x,g(\theta,x)):\theta\in B(\delta),x\in D\}$ is Borel. To see this, let $I_{j}$ be an enumeration of all open intervals with rational endpoints in $\bR$. Then $(\theta,x,y)\not\in \Gamma$ if and only if there is $j$ with $y\in I_{j}^{c}$ and $g(\theta,x)\in I_{j}$, and so $\Gamma^{c}=\bigcup_{j}  g^{-1}(I_{j}) \times I_{j}^{c}$, thus $\Gamma$ is a Borel set. 

Now define $f:B(\delta)\times \bR^{d+1}\toitself$ by $f(\theta,x)=(\theta,\Theta(\theta)(x))$. Note that$f^{-1}(\theta,x)=(\theta,\Theta(\theta)^{-1}(x))$ is also continuous. Then
\begin{align*}
f(\Gamma)
& =\{(\theta,\Theta(\theta)(x,g(\theta,x))):\theta\in B(\delta),x\in D\}\\
& =\{(\theta,\Theta(\theta)(\xi'(\theta,x)):\theta\in B(\delta),x\in D\}\\
& =\{(\theta,\xi(\theta,x)):\theta\in B(\delta),x\in D_{\theta}\}=G.
\end{align*}
Since $f$ is a homeomorphism, $G$ is also a Borel set.
\end{proof}
\def\cprime{$'$}

\end{document}